\documentclass[twoside,twocolumn]{article}

\usepackage{blindtext} 

\usepackage[sc]{mathpazo} 
\usepackage[T1]{fontenc} 
\usepackage{microtype} 

\usepackage[english]{babel} 

\usepackage[hmarginratio=1:1,top=32mm,columnsep=20pt]{geometry} 
\usepackage[hang, small,labelfont=bf,up,textfont=it,up]{caption} 
\usepackage{booktabs} 

\usepackage{amsmath}
\usepackage{amsthm}
\usepackage{hyperref}
\usepackage{amssymb}

\usepackage{lettrine} 

\usepackage{enumitem} 
\setlist[itemize]{noitemsep} 

\usepackage{abstract} 

\usepackage{titlesec} 
\renewcommand\thesection{\Roman{section}} 
\renewcommand\thesubsection{\roman{subsection}} 
\titleformat{\section}[block]{\large\scshape\centering}{\thesection.}{1em}{} 
\titleformat{\subsection}[block]{\large}{\thesubsection.}{1em}{} 

\usepackage{fancyhdr} 
\usepackage{titling} 

\usepackage{hyperref} 
\newcommand{\R}{\mathbb{R}}
\newcommand{\E}{\mathbb{E}}

\newcommand{\1}{\mathbb{I}}
\renewcommand{\P}{\mathbb{P}}
\newcommand{\Z}{\mathbb{Z}}
\newcommand{\N}{\mathbb{N}}
\newcommand{\T}{\mathbb{T}}
\newcommand{\C}{\mathbb{C}}

\newcommand{\Cc}{\mathcal{C}}

\newcommand{\B}{\mathcal{B}}

\newtheorem{thm}{Theorem}
\newtheorem{prop}[thm]{Proposition}
\newtheorem*{Pro*}{Proposition}

\newtheorem*{que*}{Question}
\newtheorem{Question}{Question}
\newtheorem{rem}[thm]{Remark}
\newtheorem{Cor}[thm]{Corollary}
\newtheorem{fact}[thm]{Fact}
\newtheorem*{thank}{\ \ \ \bf Acknowledgment}
\newcommand{\tend}[3][]{\xrightarrow[#2\to#3]{#1}}
\newcommand{\mob}{\boldsymbol{\mu}}
\newcommand{\lamob}{\boldsymbol{\lambda}}

\newcommand{\bfu}{\boldsymbol{f}}
\newcommand{\bh}{\boldsymbol{h}}
\newcommand{\bb}{\boldsymbol{b}}
\newcommand{\bc}{\boldsymbol{c}}
\newcommand{\aaa}{\boldsymbol{a}}

\pretitle{\begin{center}\Huge\bfseries} 
\posttitle{\end{center}} 
\title{Oscillating sequences, Gowers norms and Sarnak's conjecture} 
\author{%
\textsc{e. H. el Abdalaoui}\thanks{Department of Mathematics, LMRS  UMR 6085 CNRS, Avenue de l'Universit\'e, BP.12,
76801 Saint Etienne du Rouvray - France.} \\[1ex] 
\normalsize Normandy University of Rouen,\\ 
\normalsize \href{elhoucein.elabdalaoui@univ-rouen.fr}{elhoucein.elabdalaoui@univ-rouen.fr} 
}
\date{} 

\renewcommand{\maketitlehookd}{%
\begin{abstract}
It is shown that there is an oscillating sequence of higher order which is not orthogonal to the class of 
dynamical flow with topological entropy zero. We further establish that any oscillating sequence of order $d$ is orthogonal 
to any $d$-nilsequence arising from the skew product on the $d$-dimensional torus $\T^d$.  
The proof yields that any oscillating sequence of higher order is orthogonal to 
any dynamical sequence arising from topological dynamical systems with quasi-discrete spectrum. however, 
we provide an example of oscillating sequence of higher order  
with large Gowers norms. We further obtain a new estimation of the average of M\"{o}bius function on the short interval by appealing to
Bourgain's double recurence argument.   
\end{abstract}
}

\begin{document}

\maketitle


\section{Setting and tools}
 Following \cite{Fan},\cite{Jiang}, \cite{Fan}\footnote{When this paper was almost ready, the author received a email from A. 
 Fan with attached paper \cite{Fan2} in which it is proved 
that the oscillating sequences of order $d=t.s$ is orthogonal to any dynamical sequence 
of the from 
$\Big(F(T^{q_1(n)}x,\cdots,T^{q_k(n)}x\Big)$, where $q_i(n)$, $i=1,\cdots,k$ are a polynomials of degeree at most $s$ and
$T$ is a map with topological quasi-discrete spectrum of order $t$. Here, we will mention and deduce this result from our result (see
Theorem \ref{Fan2}).
} the sequence  $\mathbf{c}=(c_n)$ is said to be oscillating sequence of order $k$ ($k \geq 1$) if 
 \begin{eqnarray}\label{l1}
  \sum_{n=1}^{N}|c_n|^\lambda=O(N)~~~~~~~~~~~~\textrm{for~~some~~} \lambda \geq 1,
 \end{eqnarray}
and for any real polynomial $P \in \R_k[z]$  of degree less than or equal to $k$ we have 
\begin{eqnarray}\label{l2}
 \frac{1}{N} \sum_{n=1}^{N}c_n e^{2 \pi i P(n)} \tend{N}{+\infty}0.
 \end{eqnarray}
The sequence is said to be higher order oscillating sequence if it is an oscillating sequence of any  higher order.\\
 
The authors in \cite{Fan} introduced this notion in order to extend the study of the M\"{o}bius-Liouville randomness law to the class of 
oscillating sequences. They propose to replace the M\"{o}bius function by any higher order oscillating sequence in 
the M\"{o}bius-Liouville randomness law. 
But, as we will establish here, there is an oscillating sequence of higher order which is not orthogonal to the class of 
dynamical flow with topological entropy zero. The authors therein gives also an example 
of higher order oscillating sequences by appealing to the deep classical result of
Kahane on a subnormal random independent variables \cite{kahane}.\\

Here, we strengthen their result by showing that
if $(X_n)$ is a sequence of independent random variables with common mean 
zero and uniformly $L^p$-norm bounded for some $p>2$ then almost surely the sequence $X_n(\omega)$ is higher order oscillating.
Consequently, the sequence $(X_n(\omega))$ is almost surely higher order oscillating sequence if $X_n$ is subnormal, for each $n$.\\

We remind that the M\"{o}bius-Liouville randomness law \cite{Kowalski} 
assert that for any ``reasonable" sequence of complex numbers $(a_n)$ we have
\begin{eqnarray}\label{rlaw1}
\frac{1}{N}\sum_{n=1}^{N}\lamob(n) a_n \tend{N}{+\infty}0,
\end{eqnarray}
where $\lamob$ is the Liouville function given by 
\begin{equation*}\label{Liouville}
\lamob(n)= \begin{cases}
 1 {\rm {~if~}} n=1; \\
(-1)^r  {\rm {~if~}} n
{\rm {~is~the~product~of~}} r \\{\rm {~not~necessarily~distinct~prime~numbers}}. 
\end{cases}
\end{equation*}

Applying Chowla-Batman trick \cite{Chowla-Batman}, the \linebreak Liouville function can be replaced in \eqref{rlaw1} by 
the M\"{o}bius function $\mob$. We remind that the M\"{o}bius function is defined by 

\begin{equation*}\label{Mobius}
\mob(n)= \begin{cases}
 \lamob(n) {\rm {~if~}} n \rm{~is~not~divisible~by~}\\
 \rm{the~square~of~any~ prime}; \\
0  {\rm {~~~~if~not.}}
\end{cases}
\end{equation*}

In his seminal paper \cite{Sarnak}, P. Sarnak consider  
the M\"{o}bius-Liouville randomness law for a class of deterministic sequences which arise from topological dynamical system with topological entropy 
zero. Precisely, the sequence $(a_n)$ is  given by 
$a_n=f(T^nx)$, for any $n \geq 1$, where $T$ is homeomorphism acting on a compact space $X$ with topological entropy zero, 
$f$ is a continuous function on $X$ and $x$ a point in $X$.\\

This is nowadays known as Sarnak's conjecture. At now, as far as the author is aware, 
this conjecture was established only for many particular case of zero topological entropy
dynamical systems (see \cite{Tao-Log} and the reference therein, see also \cite{Hou-Survey}).\\

In particular, Liu \& P. Sarnak proved that Sarnak's conjecture holds for an affine linear map of nilmanifold \cite{Liu-Sarnak}  
by applying a slightly strengthen version of Green-tao's theorem \cite{Green-Tao}
combined with a classical result from \cite{Dani}.\\

Here, we will apply same classical ingredients to establish that 
the higher order oscillating sequences are orthogonal to the wide class of nilsequences and to the affine linear maps on the Abelian group. 
Indeed, our proof yields that
the higher order oscillating sequences are orthogonal to any dynamical sequence $(f(T^nx))$ provided that $T$ has a quasi-discrete spectrum. 

At this point one may ask if the previous result can be extended to the all nilsequences. Since, as pointed out by W. parry \cite{Parry},
''the nilflows and nilmanifold unipotent affines should be viewed as models generalizing the models defining quasi-discrete spectra''
.\\

We answer this question by establishing that there is an almost nilsequence which is higher order oscillating. It follows that 
there is a higher order oscillating  sequence with high Gowers norms. We thus get that the notion of 
higher order oscillating sequence is not adapted to generalize the spirit of the  M\"{o}bius-Liouville randomness law 
and Sarnak's conjecture.

Although, the orthogonality of the higher order oscillating sequences and the quasi-discrete spectrum is in the spirit of Liu \& P. Sarnak's result, 
since the M\"obius function is higher order oscillating sequence by 
Hua's theorem \cite{Hua}. Of Course, in the particular case of $G=\R^d$ and $\Gamma=\Z^d$, the proof yields that 
the oscillating sequence of order 
$d$ are orthogonal to the standard homogeneous space $(\T^d,T)$, where $T$ is an affine map.  We thus get the result of Jiang \cite{Jiang}.

We remind that the dynamical system $(X,\B,\mu,T)$ is said to have a measurable quasi-discrete spectra if the closed linear subspace spanned by 
$H=\bigcup_{n \geq 0} H_n$ is all $L^2(X,\mu)$, where  $H_0$ is the set of the constant complex valued function of modulus $1$, and for any 
$n \geq 1$, $H_n=\big\{f \in L^2(X,\mu)~~:~~|f|=1{\rm{~~a.e. and}}~~\frac{f \circ T}{f} \in  H_{n-1}\big\}.$ If
for some $d \geq 1$,  $H_d=H_{d+1}$ we say that $T$ has a discrete-spectrum of order $d$.\\

This class was defined and studied by L. M. Abramov \cite{Abramov}. 
Subsequently, F. Hahn \& W. Parry introduced and studied the notion of quasi-discrete spectrum in the topological dynamics 
for a homeomorphism $T$ of a compact set $X$ that is assumed to be completely minimal (that is, all its powers are minimal)
\cite{Parry1}, \cite{Parry2}. 
The quasi-eigen-functions are assumed to be continuous and separate the points of $X$. Therefore, by the Stone-Weierstrass theorem, 
the subalgebra generated by the quasi-eigen-functions is dense in $C(X)$. Ten years later, R. J. Zimmer shows that a 
totally ergodic system $(X,\B,T,\mu)$ has quasi-discrete spectrum if and only if it is distal and isomorphic to a totally 
ergodic affine transformation on a compact connected Abelian group $(G,S)$, that is,  $S~~:~~G \rightarrow G$ has the form 
$Sx=x_0Ax,$ where $A~~:~~G \rightarrow G$ is an automorphism of the group $G$ and $x_0 \in G$ \cite{Zimmer}.\\

Applying some algebraic arguments, one can define for any $n \geq 1$ the subgroup $G_n=\ker(\Lambda^n)$ where $\lambda$ 
is the derived homomorphism on the multiplicative
group $C(X,\T)=\big\{ f \in C(X)~~:~~|f|=1\big\}$ given by $\lambda(f)=\frac{f\circ T}{f}=f\circ T. \overline{f}$ and
$G_0=\{1\}$.
Therefore $G=\bigcup_{n \geq 1}G_n$ is an Abelian group and $\Lambda$ is a quasi-nilpotent homomorphism on it. We remind that 
$\Lambda$ is called nilpotent if $G=G_n$ for some $n$ and quasi-nilpotent if $G=\bigcup_{n \geq 1}G_n$. Notice also that 
the subspace of the invariant continuous functions is  the subspace of the constant functions $\C.G_0$ by minimality. We further have 
by the binomial theorem{\footnote{There is an analogy between this formula and Hall-Petresco  identity for the nilpotent groups (see \cite[p.118]{TaoHfourier}.}
\begin{eqnarray}\label{binomial}
 f\circ T^n= \prod_{j=0}^{n}\big(\Lambda^j(f)\big)^{\binom{n}{j}}
\end{eqnarray}
for each $f \in C(X,\T)$, where the binomial coefficients $\binom{n}{j}$ are defined by 
$$\binom{n}{j}=\begin{cases}
                \frac{n(n-1)\cdots(n-j+1)}{j!} {\rm {~if~}} 0 \leq j \leq n\\
                0{\rm {~if~not}} 
               \end{cases}
$$
The elements of $G_n$ are called quasi-eigenvectors of order $n-1$ and $G$ is the group of all quasi-eigenvectors. By putting
$H_n=\Lambda(G_{n+1})$, we see that the elements of $G_n$ are precisely the unimodular solutions $f$ of the equation
$\Lambda(f)=g$, where $g \in G_{n-1}$. The elements of the subgroup $H_n$ are called a quasi-eigenvalue of order $n-1$ and
$$H=\bigcup_{n \geq 0}H_n,$$
is the group of all quasi-eigenvalues. Obviously, $\iota~~:f \in G_1 \mapsto f(x_0) \in \T$ where $x_0 \in X$ is an
isomorphism of groups. A triple $(H,\Lambda,\iota)$ is called the signature of the dynamical system $(X,T)$. 
According to Hahn-Parry's Theorem \cite{Hahn-Parry}, if $(X,T)$ is totally minimal topological system with quasi-discrete spectrum
and signature $(H,\Lambda,\iota)$, then $(X,T)$ is isomorphic to the affine automorphism system $(\widehat{H},\phi^*,\eta)$ where 
$\widehat{H}$ is the dual group of $H$, $\phi(h)=h\Lambda(h)$, for $h \in H$, and $\eta$ denotes any homomorphic extension of 
$\eta : H_1 \longrightarrow \T$ to all of $H$.\\

The popular example of maps with quasi-discrete spectrum is given by the following transformation of
the $d$-dimensional torus of the form
\[
  T(x_1,\ldots,x_d)=(x_1+\alpha, x_2+x_1, \ldots, x_{d}+x_{d-1}).
\]
This transformation is an affine transformation, it can be written as $x\mapsto Ax+b$ where $A=[a_{ij}]_{i,j=1}^d$ is the matrix 
defined by $a_{1,1}:=1, a_{i-1,i}=a_{ii}:=1,~~i=2,\cdots,d$ and all other coefficients equal to zero, and $b:=(\alpha,0,\ldots,0)$.
Taking again $\alpha$ irrational, $(\T^d,T)$ is a uniquely ergodic dynamical system, and it is totally ergodic with respect to the 
Haar measure on $\T^d$, which is the unique invariant measure \cite{Fu}. More generally, H. Hoare and W. Parry established that 
if $T$ is a minimal affine transformation of a compact connected abelian group $X$, that is, $T(x)=a.A(x),~~x \in X$, 
where $A$ is an automorphism of $X$ and $a \in X$ then $T$ has quasi-discrete spectrum \cite{Hoare-Parry}. 
For a recent exposition and analysis of the subject, we refer the reader to
\cite{Haas}.\\

Let us further remind that the authors in \cite{Ab-Ka-Le} proved that Sarnak's conjecture holds for any 
uniquely ergodic model of a dynamical system with quasi-discrete spectrum. The proof is based on the joining property of the the powers called Asymptotical Orthogonal Powers (AOP). 
This later property insure  that Katai-Bourgain-Sarnak-Ziegler criterion holds.\\

We end this section by pointing out that  the M\"{o}bius-Liouville randomness law can be seen as a weaker version of the 
following notion of independence introduced by Rauzy in \cite{Rauzy}.\\

Let $X,Y$ be two metric spaces, we say that the sequence 
$(x_n) \subset X $ and $(y_n) \subset X $ are independent if for any continuous functions $f \in C(X)$ and 
$g \in C(Y)$ we have 
\begin{eqnarray*} 
&&\Big|\frac{1}{N}\sum_{n=0}^{N-1}f(x_n)g(y_n)\\
&&-\Big(\frac{1}{N}\sum_{n=0}^{N-1}f(x_n)\Big)
\Big(\frac{1}{N}\sum_{n=0}^{N-1}g(y_n)\Big)\Big|\tend{N}{+\infty}0.
\end{eqnarray*}


\section{the main results}

We start by stating our first main result.

\begin{thm}\label{Mainofmain}
There exist a dynamical system $(X,T)$ with topological entropy zero and a higher order oscillating sequence which is not orthogonal
to $(X,T)$
\end{thm}




Our second main result is the following

\begin{thm}\label{main1}Let $(X_n)$ be any sequence of independent random variables $(X_j)$ such that
$\E(X_j)=0$ and $\sup_{j \geq 0}\E(|X_j|^p) <+\infty$, for some $p>2$. Then the sequence $c_n=X_n(\omega)$ is 
almost surely higher order oscillating sequence. 
\end{thm}

It follows that if $X_n$ is a subnormal random variable for each $n \geq 1$, 
that is, $\E(e^{\lambda.X_n}) \leq e^{\frac{\lambda^2}{2}},$ for any $\lambda \in \R$. Then, we have 
\begin{Cor}[\cite{Fan}]\label{Fan1}
 Let $(X_n)$ be any sequence of independent random variables such that $X_n$ is subnormal for each $n \geq 1$. Then,
 the sequence $c_n=X_n(\omega)$ is  almost surely higher order oscillating sequence. 
\end{Cor}
For the proof of Theorem \ref{main1}, we need the following classical inequalities due to J. Marcinkiewicz and A. Zygmund.
\begin{thm}\cite{MZ}
 If $X_n$, $n=0,\cdots N,$ are independent $\C$-valued random variables
 with mean zero and finite $L^p$-norm, $p \geq 1$. Then 
 $$A_p \Big\|\sum_{n=1}^{N}X_j\Big\|_2 \leq \Big\|\sum_{n=1}^{N}X_j\Big\|_p \leq B_p \Big\|\sum_{n=1}^{N}X_j\Big\|_2,$$  
for positive constants $A_p$ and $B_p$ depending only on $p$.
\end{thm}
Marcinkiewicz-Zygmund inequalities generalize the well-known Khintchine inequalities which assert that the $L^p$-norms are equivalent
for the Rademacher variables. The proof given in \cite{MZ} is in French language.
For the more recent proof and for its extension to the martingale setting, we refer to \cite[Chap. 3. p.73]{Garcia} and \
\cite[Chap.11. p.412]{Chow}.\\

\begin{table}
\caption{the M\"{o}bius-Liouville randomness law \\vs Sarnak's conjecture.}
\centering
\begin{tabular}{llr}
\toprule
\multicolumn{2}{c}{$\{-1,1\}^\Z$, $S$ is a shift map, } \\
\cmidrule(r){1-2}
$(X,T),$ $h_{\rm{top}}(T)=0$ & $Y=O(\lamob)$&  \\
\midrule
$x_n=T^nx$ & $y_n=S^n\lamob$ & $\perp?$  \\
$f \in C(X)$ & and $\pi_0(y)=y_0$  &  \\
\bottomrule
\end{tabular}
is the $1^{\rm{th}}$ projection, $O(\lamob)$ is the orbit closure.
\end{table}

We are now able to proof Theorem \ref{main1}.
\begin{proof}[\textbf{Proof of Theorem \ref{main1}.}] Let us assume, without any loss of generality that 
$\sup_{n \geq 0}E(|X_n|^2) \leq 1$. Then, by our assumption combined with Marcinkiewicz-Zygmund inequalities it 
follows that for any $p>1$,
$$\E\Big(\Big|\frac{1}{N}\sum_{n=0}^{N-1}X_n e^{2 \pi i P(n)}\Big|^p\Big) \leq B_p \frac1{N^{\frac{p}{2}}}.$$ 
Hence, for $p>2$, we have
$$\E\Big(\sum_{N \geq 1} \Big|\frac{1}{N}\sum_{n=0}^{N-1}X_n e^{2 \pi i P(n)}\Big|^p\Big)<\infty.$$
A standard argument yields the desired property.
\end{proof}
The proof of Corollary \ref{Fan1} follows from the following well-know fact.
\begin{fact} Let $X$ be a random variable such that $\E(X)=0$. Then the following are equivalent
\begin{enumerate}
 \item There exist $c>0$ such that for any $\lambda \geq 0$, $\P\big\{|X| \geq \lambda\big\} \leq 2 \exp(-\lambda^2 c).$
 \item  There exist $c'>0$ such that for any $p \geq 1$, $\big\|X\big\|_p \leq c'\sqrt{p}.$
 \item There exist $c''>0$ such that for any $t \in \R$, $\E(\exp(tX)) \leq \exp\big(c''.t^2/2\big)$.
\end{enumerate}
\end{fact}

\begin{rem}Applying Koksma classical result combined with the well-know criterion of \linebreak van der Corput, 
S. Akiyama \& Y. Jiang established in \cite{Akiyama} that for any positive real valued
$2$-times continuously differentiable function $g$ on $(1,+\infty)$ (that is, $g(x)>0,g'(x),g''(x) \geq 0$), for any  $\alpha \in \R^*$, 
for almost all real numbers $\beta>1$ and for any real polynomials $Q$, the sequences
$$\bfu_{\alpha,\beta}(n)=\alpha \beta^n g(\beta)+Q(n)$$
is higher order oscillating. Let us notice that the proof of the previous result can be obtained in the same spirit as 
the proof of Theorem \ref{main1} by the general
principle stated in  \cite[Theorem 4.2,p.33]{NK}.
\end{rem}
Applying further the Koksma general metric criterion (Theorem 4.3 in \cite{NK}),
one can easily seen  for any  $\alpha \in \R^*$, for almost all real numbers $\beta>1$ and 
for any real polynomials $Q$, for any $\ell \ne 0$, we have
\begin{eqnarray}\label{MOMO}
&&\frac{1}{M}\sum_{M \leq m \leq 2M}\Big|\frac{1}{H}\sum_{m \leq h <m+H} e^{2 \pi i \ell \bfu_{\alpha,\beta}(h)}\Big|\nonumber
\\
&&\tend{H}{+\infty,~~ H/M \rightarrow 0 }0.
\end{eqnarray}
or, equivalently \cite{Elabdelal},  for any  $\alpha \in \R^*$, for almost all real numbers $\beta>1$, for any real polynomials $Q$
and for any increasing sequence of integers
$0=b_0<b_1<b_2<\cdots$ with $b_{k+1}-b_k\to\infty$, for almost all $\beta>1$, for any $\ell \ne 0$, we have
\begin{equation}
    \label{eq:MOMO1}
    \frac{1}{b_{K}} \sum_{k< K} \Big|\sum_{b_k\le n<b_{k+1}} e^{2 \pi i \ell \bfu_{\alpha,\beta}(n)} \Big| \tend{K}{\infty} 0.
  \end{equation}

The proof of \eqref{MOMO} follows as a consequence of the fact that the Koksma general metric criterion is valid for the short interval.
Indeed, applying Cauchy-Schwarz inequality it suffices to establish the following key inequality. For any $1<a<b$,
\begin{eqnarray*}
 &&\int_{a}^{b} \Big|\frac{1}{H}\sum_{m \leq h <m+H} e^{2 \pi i \ell \bfu_{\alpha,\beta}(h)}\Big|^2 d\beta \\
&&\leq \frac{b-a}{H}+C.\frac{\log(3H)}{|\ell|.H},
\end{eqnarray*}

for same absolute constant $C>0.$\\

This inequality is valid  by the same arguments as in the proof of Theorem 4.3 in \cite{NK}.\\

Before stating our third main result, let us remind the notion of nilsequences and some tools.

\subsection*{\bf{I. Nilsequences and nilsystems}} A sequence $(b_n)$ is said to be a $k$-step basic nilsequence if there is a nilpotent Lie group $G$ of order $k$ and a discrete co-compact 
subgroup $H$ of $G$ such that $b_n=F(T_g^nx\Gamma)$  where $T_g~~:x.\Gamma \in G/H \mapsto (gx).\Gamma$, $g \in G$, $F$ is continuous function on $X=G/H$. The homogeneous space 
$X=G/\Gamma$ equipped with the Haar measure $h_X$ and the canonical complete $\sigma$-algebra
$\B_{c}$. the dynamical system $(X,\B_{c},h_X,T_g)$
is called a $k$-step nilsystem and $X$ is a $k$-step nilmanifold. 
For a nice account of the theory of the homogeneous space we refer the reader to \cite{Dani}. Let us notice that any affine linear map on $X$ has zero entropy if and only if
if it is quasi-unipotent. So we assume here that the affine linear maps are quasi-unipotent. We further assume that  $G$  
is connected and simply-connected by Leibman's arguments \cite{Leib}. If $X$ is Abelian we say that the nilsequence $(b_n)$ is an Abelian nilsequence.\\

Let us further point out that $(b_n)$ is any element of
$\ell^{\infty}(\Z)$, the space of bounded sequences, equipped with uniform norm
$\displaystyle \|(a_n)\|_{\infty}=\sup_{n \in \Z}|a_n|$. A $k$-step nilsequence, is a uniform limit of basic $k$-step nilsequences. We remind that 
$X$ is said to be a $k$-step nilmanifold.\\

Let $G$ be a nilpotent Lie group with a cocompact lattice $\Gamma$ and a $\Gamma$-rational filtration $G_{\bullet}$ of
length $l$, so that $\Gamma_i = \Gamma \cap G_i$ is a cocompact lattice in $G_i$ for each $i = 1,\cdots, l$. 

We remind that the sequence of subgroups $(G_n)$ of $G$ is a filtration if  $G_1=G,$ $G_{n+1}\subset G_{n},$ and
$[G_n,G_p] \subset G_{n+p},$ where $[G_n,G_p]$ denotes the subgroup of $G$ generated by the commutators $[x,y]=
x~y~x^{-1}y^{-1}$ with $x \in G_n$ and $y \in G_p$.
The lower central filtration is given by $G_1=G$ and $G_{n+1}=[G,G_n]$. It is well know that the lower central filtration allows to construct a Lie algebra $\textrm{gr}(G)$ over the ring $\Z$ of integers. $\textrm{gr}(G)$ is called a
graded Lie algebra associated to $G$ \cite[p.38]{Bourbaki2}. The filtration is said to be of length $l$ if
$G_{l+1}=\{e\},$ where $e$ is the identity of $G$.\\

An example of $k$-nilsequence is given by a continuous function $F$ which satisfy
$F(g_ly)=\chi(g_l\Gamma_l) F(y)$, for any $y \in G$, $g_l \in G_l$ and $\chi \in \widehat{G_l/\Gamma_l}$ where $\widehat{G_l/\Gamma_l}$ is the dual group of 
the Abelian group ${G_l/\Gamma_l}$. The function $F$ is called a vertical nilcharater. It follows that the Hilbert space $L^2(X,h_X)$ can be 
decomposed into a sum of $G$-invariant orthogonal Hilbert subspaces. Let us also point out that the quasi-unipotent 
case can be reduced to the unipotent case, and for more details on the Fourier analysis theory on the nilspaces we refer to \cite{TaoHfourier}.\\

\paragraph{A. Furstenberg's argument and Bergelson-Leibman's observation.}
Acoording to Furstenberg's argument \cite[p.23]{Fur}, if $P(n) \in \R_d[X]$, $d \geq 1$, then 
the sequence $e^{2 \pi i P(n)}$ is a dynamical sequence. 
Indeed, write 
\begin{eqnarray*}
 P(n)&=&a_0+a_1n+a_2n^2+\cdots+a_dn^d \\
&=&b_0+b_1n+b_2\binom{n}{2}+\cdots+b_d\binom{n}{d},
\end{eqnarray*}
and for any $x\in \T^d$ put $Tx=Ax+\bb,$ where  $\bb=(b_d,\cdots,b_1)$ and the matrix $A=\big(A[i,j]\big)_{i,j=1}^d$ is 
defined by $A[1,1]:=1, A[i-1,i]=A[i,i]:=1,$ for  $i=2,\cdots,d$ and all other coefficients equal to zero. Then, a straightforward 
computation by  induction on $n$ yields
$$\chi_d(T^n(x_0))=e^{2 \pi i P(n)},$$
where $x_0=(0,\cdots,0,b_0)$ and $\chi_d(x)=e^{2\pi i x_d}.$\\

Furthermore, by Bergelson-Leibman's argument \cite{Leib}, the map $T$ can be viewed as a nilrotation. Indeed, 
Let $G$ be the group of upper trianglar matrix $T=\big(T[i,j]\big)_{i,j=1}^{d+1}$ with $T[i,i]=1,$ $i=1,\cdots, d+1$,
$T[i,j] \in \Z$, $1 \leq i < j \leq d$ and $T[i,d+1] \in \R$, for $i=1,\cdots,d$. Consider $\varGamma$ the subgroup of $G$ consisting
of the matrices with integer entries. Then $G$ is a nilpotent non-connected Lie group with $X=G/\varGamma \simeq \T^k$, and define 
the nilrotation $T_g$ on $X$ by $T_g(x)=gx$ where $g[i,i]=1$, for $i=1,\cdots,d+1$, $g[i-1,i]=1$, $i=2,\cdots,d$, 
$g[j,d+1]=b_j$, for $j=1,\cdots,d$ and all other coefficients equal to zero. We thus get that the nilrotation $T_g$ is isomorphic to 
the skew product $T$ defined on $\T^d$.\\

We can thus consider the dynamical sequence $F(T^nx)$, where $F$ is a continuous and $T$ is a skew product on the $d$-torus as a 
$d$-nilsequence up to isomorphism.\\

Following \cite{Tao-Green-ZU4}, the $1$-bounded sequence $(a(n))$ is said to be an almost nilsequence of degree $s$ with complexity $O_M(1),$ where $M>1$ is a given 
complexity parameter, if for any $\varepsilon>0$ there is a nilsequence $(a_{\varepsilon}(n))$ with complexity $O_{s,\varepsilon,M}(1)$ such that 
$$\frac{1}{N}\sum_{n=1}^{N}\big|a(n)-a_{\varepsilon}(n)\big|<\varepsilon.$$
Roughly speaking, the $L^1$ closure of the space of the nilsequences of degree $s$ is the  space of the almost nilsequences of degree $s$. 
In \cite{Tao-Green-ZU4}, the authors gives various examples of almost nilsequences of degree $s \leq 3$. 

Applying the fundamental tools of this theory combined with some ingredients and results due to A. Liebman, V. Bergelson \& A. 
Liebman \cite{Leib}, we can establish the following 

\begin{thm}\label{main2}
 There exist a higher order oscillating sequence wich is an almost nilsequence.
\end{thm}

However,  by Lemma 3.4 from \cite{Tao-Green-ZU4}, it is easy to check that the oscillating sequences of order $1$ are orthogonal to the 
almost nilsequence of degree $1$. Moreover,  we can easily check the following.

\begin{thm}\label{nilsequence}
 The sequence $(c_n)$ of oscillating order $d$ is orthogonal to any nilsequence of order $d$ arising from skew product on the 
 $d$-dimensional torus $\T^d$.
\end{thm}
\begin{proof}
 A straightforward by Furstenberg's argument (see subsection A.).
\end{proof}

We thus get, by applying the classical density argument, the following
\begin{Cor}\label{nilsequenceII}
 The sequence $(c_n)$ of oscillating order $d$ is orthogonal to any affine transformation on the 
 $d$-dimensional torus $\T^d$.
\end{Cor}

Applying the same reasoning, we have
\begin{thm}\label{quasi}
 The sequence $(c_n)$ of oscillating order $d$ is orthogonal to any quasi-discret system of order $d$.
\end{thm}
\begin{proof}
By the density argument, it suffices to establish the orthogonality 
for a functions $f \in G_k,$ $k \geq 0$. Let $f \in G_{k+1}\setminus G_k,$ then, by \eqref{binomial}, for $n \geq k$, $x \in X$, we have 
\begin{eqnarray*}
 f(T^nx)&=&\prod_{j=0}^{k}\Big(\Lambda^jf\Big)^{\binom{n}{j}}(x)\\
 &=&\prod_{j=0}^{k}\Big(\Lambda^jf(x)\Big)^{\binom{n}{j}}\\
 &=&f(x)e^{2 \pi i P(n)},
\end{eqnarray*}
Where $P(n)=\displaystyle \sum_{j=1}^{k}\binom{n}{j} \theta_j$ and  $ \theta_j \in \R$ is such that
$\big(\Lambda^jf\big)(x)=e^{2 \pi i \theta_j}$. Therefore, by \eqref{l2}, it follows that 
$$\frac{1}{N}\sum_{n=0}^{N}c_nf(T^nx) \tend{N}{+\infty}0.$$
This achieve the proof of the theorem.
\end{proof}
At this point let us mention the result of A. Fan \cite{Fan2} and gives its proof.

\begin{thm}[\cite{Fan2}]\label{Fan2}~The sequence $(c_n)$ of oscillating order $d=t.s,$ $t,s \in \N^*$ is orthogonal to any
dynamical sequence of the form $F\big(T^{q_1(n)}x,\cdots,T^{q_k(n)}x\big)$, where $T$ is 
a homeomorphic map on a compact set $X$ with quasi-discret spectrum of order $t$, 
$F$ is a continuous function on $X^k$ and $q_i(n),i=1,\cdots,k$ are a polynomials of
degeree at most $s$.
\end{thm}
\begin{proof}By density argument  it is suffices to check the orthogonality for the function $F$ of the form 
$$F(x_1,x_2,\cdots,x_k)=f_1(x_1).f_2(x_2)\cdots f_k(x_k),$$
where $f_i,i=1,\cdots,k$ are a eigen-functions. In the same manner as before, we apply \eqref{binomial} to get
\begin{eqnarray*}
 &&\frac{1}{N}\sum_{n=0}^{N-1}c_n F\big(T^{q_1(n)}x,\cdots,T^{q_k(n)}x\big)\\
&=&F(x,\cdots,x)\frac{1}{N}\sum_{n=0}^{N-1}c_n e^{2\pi i Q(n)},
\end{eqnarray*}
where $Q(n)$ is a polynomials with at most $d$ degree. We thus conclude that 
$$
\frac{1}{N}\sum_{n=1}^{N-1}c_n F\big(T^{q_1(n)}x,\cdots,T^{q_k(n)}x\big) \tend{N}{+\infty}0.$$
The proof of the theorem is complete.
\end{proof}

The poof of \eqref{MOMO} can be adapted to obtain the following
\begin{thm} For any dynamical flow $(X,T)$, for any continuous function $f$, 
for any $x\in X$, for any $\alpha \neq 0$ and for almost all $\beta>1$, we have
\begin{eqnarray*}
 &&\frac{1}{M}\sum_{M \leq m \leq 2M}\Big|\frac{1}{H}\sum_{m \leq h <m+H} f(T^hx) e^{2 \pi i \bfu_{\alpha,\beta}(h)}\Big| \\
&&\tend{H}{+\infty,~~ H/M \rightarrow 0 }0. 
\end{eqnarray*}
\end{thm}
We thus get 
\begin{Cor}
For any nilsequence $(b_n) \subset \C$, for any $\alpha \neq 0$ and for almost all $\beta>1$, we have
\begin{eqnarray*}
 &&\frac{1}{M}\sum_{M \leq m \leq 2M}\Big|\frac{1}{H}\sum_{m \leq h <m+H} b_h e^{2 \pi i \bfu_{\alpha,\beta}(h)}\Big|
\\
&& \tend{H}{+\infty,~~ H/M \rightarrow 0 }0.
\end{eqnarray*}

\end{Cor}



Notice that the tools applied here allows one to obtain a generalization of a results of \cite{Fan}, 
\cite{Jiang} and \cite{Huang} to the case of quasi-discrete spectrum. Let us notice further that in \cite{ab-k-m-r} the authors 
established a criterion of M\"{o}bius disjointness for uniquely ergodic systems. This criterion yields 
that Sarnak's conjecture holds for any 
topological model for a class of uniquely ergodic dynamical systems including quasi-unipotent nilsystems. As in \cite{Huang}, 
the fondamental ingredients is based on the short intervall orthogonality of the M\"{o}bius function to the rotation on the circle. 
This later result extends Theorem of Davenport-Hua \cite{Davenport}, \cite[Theorem 10.]{Hua} (see also \cite{Liu-Sarnak}) which say that, 
for any $k \geq1$, for any $\epsilon>0$, 
$$\sup_{\theta}\Big|\frac1N\sum_{n=1}^{N}\mob(n)e^{2 \pi i n^k \theta}\Big| \leq \frac{C_\epsilon}{\log(N)^\epsilon}.$$
For the recent proof of the short intervall orthogonality of the M\"{o}bius function to the rotation on the circle, we refer
to \cite{HuangII} and the references therein. Let us 
notice that  in the spirit of Davenport-hua's theorem, \linebreak E. H. El Abdalaoui and X. Ye proved in \cite{elabalYe} that for any 
For any integer
$N\geq 2$ and  for any 
$\epsilon>0$, we have 
$$\sup_{a \in \ell^{\infty},\big\|a\big\|_{\infty}\leq 1}
 \Big(\frac{1}{N}\sum_{n=1}^{N}\Big|\frac{1}{N}\sum_{l=1}^{N}\mob(l+n)a_l\Big|\Big)$$$$
 \leq \frac{C_\epsilon}{\big(\log(2N)\big)^{\epsilon}},
$$
where $C_\epsilon$ is an absolutely constant which depend only on $\epsilon$. 
The proof is based on Bourgain's observation and Parseval-Bessel inequality.\\

We remind that the authors in \cite{Huang} used essentially the 
the following estimation due to Matom\"aki, Radziwi\l\l~and Tao \cite{Ma-Ra-Ta}: for any integer
$N,L \geq 10$ , for any 
$\epsilon>0$, we have 
\begin{eqnarray*}
 &&\sup_{\beta \in \T}\Big(\frac{1}{N}\sum_{n=1}^{N}\Big|\frac{1}{L}\sum_{l=1}^{L}\mob(l+n)e^{2 \pi i \beta l}\Big|\Big)\\
 &&\leq C_\epsilon \Big(\frac{1}{\big(\log(N)\big)^{\epsilon}}+\frac{\log(\log(L))}{\log(L)}\Big),
\end{eqnarray*}
where $C_\epsilon$ is an absolutely constant which depend only on $\epsilon$.\\ 

Here in the same spirit as in \cite{elabalYe}  we establish the following\\

\begin{thm}\label{Huang-plus} For any integers
$N,L \geq 2$, $k \geq 1$ and for any 
$\epsilon>0$, we have 
\begin{eqnarray*}
&& \sup_{\beta \in \T}\Big(\frac{1}{N}\sum_{n=1}^{N}\Big|\frac{1}{L}\sum_{l=1}^{L}\mob(l+n)e^{2 \pi i \beta l^k}\Big|\Big)\\ 
&\leq &\frac{C_\epsilon}{\big(\log(L+N)\big)^{\epsilon}}\big(L+N\big)\sqrt{\frac{1}{NL}},
\end{eqnarray*}
where $C_\epsilon$ is an absolutely constant which depend only on $\epsilon$.\\ 
\end{thm}
Notice that in the particular case $N=L$ we get
\begin{eqnarray*}
 &&\sup_{\beta \in \T}\Big(\frac{1}{N}\sum_{n=1}^{N}\Big|\frac{1}{L}\sum_{l=1}^{L}\mob(l+n)e^{2 \pi i \beta l^k}\Big|\Big)\\
 &\leq& \frac{C_\epsilon}{\big(\log(L)\big)^{\epsilon}}\frac{2L}{L} \leq \frac{2.C_\epsilon}{\big(\log(2L)\big)^{\epsilon}}
\end{eqnarray*}
where $C_\epsilon$ is an absolutely constant which depend only on $\epsilon$. The proof of Theorem \ref{Huang-plus} is based on Bourgain's observation \cite{Bourgain-D} 
and the classical Bessel-Parseval inequality.\\

Before given the proof of Theorem \ref{Huang-plus}, let us present the proof of Theorem \ref{main2} .

\begin{proof}[{\textbf{Proof of Theorem \ref{main2}.}}]Let $\alpha$ and $\beta$ be a rationally independent irrational numbers 
and consider the sequence  $u(n)=n\alpha [n \beta]$ mod $1$. Then, by Lemma 3.6 from \cite{Tao-Green-ZU4}, $(u(n))$ is an almost nilsequence. 
Furthermore, by Bergelson-Leibman's result \cite{Leib},
for any polynomial $P \in \R[x]$, we have
$$\frac{1}{N}\sum_{n=1}^{N}\exp\Big( 2\pi i \big(u(n)-P(n)\big)\Big) \tend{N}{+\infty}0.$$
 Therefore the sequence $(u(n))$ mod $1$ is a higher oscillating sequence. The proof of the theorem is complete.
\end{proof}

Applying the same method as in \cite{Leib}, one can exhibit more large class of almost nilsequences which are a higher oscillating sequences.\\

\begin*{Proof of Theorem \ref{Huang-plus}.} Let $N, L \geq 2$, $k \geq 1$ and $\epsilon>0$. Then, by Cauchy-Schwarz inequality, we have
\begin{eqnarray*}
&& \frac{1}{N}\sum_{n=1}^{N}\Big|\frac{1}{L}\sum_{l=1}^{L}\mob(l+n)e^{2 \pi i \beta l^k}\Big| \\
 &\leq& \Big(\frac{1}{N}\sum_{n=1}^{N}\Big|\frac{1}{L}\sum_{l=1}^{L}\mob(l+n)e^{2 \pi i \beta l^k}\Big|^2\Big)^{\frac12}.
\end{eqnarray*}
Therefore it suffices to estimate the RHS. For that, observe that by the classical Hilbert space analysis trick  we have
\begin{eqnarray*}
\sup_{\|\bb\|_2=1}\Big|\frac{1}{N}\sum_{n=1}^{N}a_n \overline{b_n}\Big|=\|\aaa\|_2,
\end{eqnarray*}
where $\bb=(b_n)_{n=1}^{N},
\aaa=(a_n)_{n=1}^{N},$ and $\|.\|_2$ is the classical euclidienne norm on $\C^N$. Let $\bb \in \C^N$ such that $\|\bb\|_2=1$. We thus need 
only to esitmate the following
$$ \Big|\frac{1}{N}\sum_{n=1}^{N} \Big(\frac{1}{L}\sum_{l=1}^{L}\mob(l+n)e^{2 \pi i \beta l^k}\Big)
\overline{b_n}\Big|.$$
But 
\begin{eqnarray*}
 &&\frac{1}{N}\sum_{n=1}^{N}\Big(\frac{1}{L}\sum_{l=1}^{L}\mob(l+n)e^{2 \pi i \beta l^k}\Big)\overline{b_n}\\
 &=&\frac{1}{L}\sum_{l=1}^{L}\Big(\frac{1}{N}\sum_{n=1}^{N}\mob(l+n)\overline{b_n}\Big).e^{2 \pi i \beta l^k}\Big),
\end{eqnarray*}
and by the triangle inequality 
\begin{eqnarray*}
 &&\Big|\frac{1}{N}\sum_{n=1}^{N}\Big(\frac{1}{L}\sum_{l=1}^{L}\mob(l+n)e^{2 \pi i \beta l^k}\Big)\overline{b_n}\Big|\\
 &\leq& \frac{1}{L}\sum_{l=1}^{L}\Big|\frac{1}{N}\sum_{n=1}^{N}\mob(l+n)\overline{b_n}\Big|
\end{eqnarray*}
Whence, once again by Cauchy-Shwarz inequality, we have
\begin{eqnarray*}
 &&\Big|\frac{1}{N}\sum_{n=1}^{N}\Big(\frac{1}{L}\sum_{l=1}^{L}\mob(l+n)e^{2 \pi i \beta l^k}\Big)\overline{b_n}\Big|\\
 &\leq& \Big(\frac{1}{L}\sum_{l=1}^{L}\Big|\frac{1}{N}\sum_{n=1}^{N}\mob(l+n)\overline{b_n}\Big|^2\Big)^{\frac12}.
 \end{eqnarray*}
We further have  
\begin{eqnarray*}
 &&\sum_{l=1}^{L}\Big|\frac{1}{N}\sum_{n=1}^{N}\mob(l+n)\overline{b_n}\Big|^2\\ 
 &&\leq 
 \sum_{l=1}^{L}\Big| \int_{\T}\Big(\frac{1}{N}\sum_{m=1}^{L+N}\mob(m)z^{m}\Big) \Big(\sum_{n=1}^{N}\overline{b_n} 
 z^{-n}\Big) z^{-l}
 dz \Big|^2 
\end{eqnarray*}
Consequently, by Bessel-Parseval inequality, we obtain
\begin{eqnarray*}
&& \sum_{l=1}^{L}\Big|\frac{1}{N}\sum_{n=1}^{N}\mob(l+n)\overline{b_n}\Big|^2\\
 &\leq& \int \Big|\frac{1}{N}\sum_{m=1}^{L+N}\mob(m)z^{m}\Big|^2 \Big|\sum_{n=1}^{N} \overline{b_n} z^{-n}\Big|^2 dz\\
 &\leq& \sup_{z \in \T} \Big|\frac{1}{N}\sum_{m=1}^{L+N}\mob(m)z^{m}\Big|^2 \Big| .
 \Big(\sum_{n=1}^{N} |b_n|^2 \Big)\\
 &\leq& \sup_{z \in \T} \Big|\frac{1}{N}\sum_{m=1}^{L+N}\mob(m)z^{m}\Big|^2.N
\end{eqnarray*} 
Furthermore, by applying Davenport-Hua theorem, we get
\begin{eqnarray*}
&& \sum_{l=1}^{L}\Big|\frac{1}{N}\sum_{n=1}^{N}\mob(l+n)\overline{b_n}\Big|^2\\
 &\leq& \frac{C_\epsilon^2}{\big(\log(L+N)\big)^{2 \epsilon}}.\Big(\frac{L+N}{N}\Big)^2 .N,
\end{eqnarray*} 
where $C_\epsilon$ is an absolutely constant which depend only on $\epsilon$. We thus conclude that we have
\begin{eqnarray*}
&& \sup_{\beta \in \T}\Big(\frac{1}{N}\sum_{n=1}^{N}\Big|\frac{1}{L}\sum_{l=1}^{L}\mob(l+n)e^{2 \pi i \beta l^k}\Big|\Big)\\
 &\leq& \frac{C_\epsilon}{\big(\log(L+N)\big)^{\epsilon}}.\big(L+N\big).\sqrt{\frac{1}{NL}}.
\end{eqnarray*}
This finish the proof of the Theorem.
\end*{\hfill $\Box$}

\begin*{\textbf{Remark.}} Let us 
notice that our proof yields that for any  integer $N,L \geq 2$ and for any 
$\epsilon>0$, we have 
$$\sup_{a \in \ell^{\infty},\big\|a\big\|_{\infty}\leq 1}
 \Big(\frac{1}{N}\sum_{n=1}^{N}\Big|\frac{1}{L}\sum_{l=1}^{L}\mob(l+n)a_l\Big|\Big)$$$$
 \leq \frac{C_\epsilon}{\big(\log(L+N)\big)^{\epsilon}}.\big(L+N\big).\sqrt{\frac{1}{NL}},
$$
where $C_\epsilon$ is an absolutely constant which depend only on $\epsilon$.\\ 
\end*{}

\subsection*{\bf{II. Gowers norms.}}
The notion of nilsequences is closely related to the notion of Gowers norms. These norms were introduced by T. Gowers in \cite{Gowers}
. Therein T. Gowers produced a new proof of Szemer\'edi's theorem. Nowadays, it is turn out that Gowers uniform norms are 
tools of great use in additive number theory, arithmetic combinatorics and ergodic theory.\\

Let $d \geq 1$ and $C_d=\{0,1\}^d$. Let $(G,+)$ be a local compact Abelian group. If $\bh \in G^d$ and $\bc \in C_d$, then 
$\bc.\bh=\sum_{i=1}^{d}c_ih_i.$  Let ${\big(f_{\bc}\big)}_{\bc \in C_d}$ be a family of bounded functions that are compactly supported, 
that is, for each $\bc \in C_d$, $f_{\bc}$ is in $L_c^{\infty}(G)$ the subspace of functions that are compactly supported. 
The Gowers inner product is given by 
$$\Big\langle \big(f_{\bc}\big)\Big \rangle_{U^d(G)}=\int_{G^{d+1}}\prod_{\bc \in C_d}\Cc^{|\bc|}f_{\bc}(g+\bc.\bh) d\bh dg
,$$
where $|\bc|=\bc.\boldsymbol{1}$, $\boldsymbol{1}=(1,1,\cdots,1) \in C_d$ and $\Cc$ is the conjugacy anti-linear operator.
If all $f_{\bc}$ are the same function $f$ then the Gowers uniform norms of $f$ is defined by
$$\big|\big|f\big|\big|_{U^d(G)}^{2^d}=\Big\langle \big(f\big)\Big \rangle_{U^d(G)}.$$
The fact that $\big|\big|.\big|\big|_{U^d(G)}$ is a norm for $d \geq 2$ follows from the following 
generalization of Cauchy-Bunyakovski-Schwarz inequality for the Gowers inner product.
\begin{prop}(Cauchy-Bunyakovskii-Gowers-Schwarz inequality)\label{Tao}
 $$ \Big\langle \big(f_c\big)\Big \rangle_{U^d(G)} \leq \prod_{c \in C_q}\big|\big|f_c\big|\big|_{U^d(G)}.$$
\end{prop}
The proof of Cauchy-Bunyakovskii-Gowers-Schwarz inequality can be obtain easily by applying inductively Cauchy-Bunyakovskii-Schwarz inequality.
Indeed, it easy to check that we have
$$ \Big\langle \big(f_c\big)\Big \rangle_{U^d(G)} \leq \prod_{j=0,1}
\Big|\langle \big(f_{\pi_{i,j}}(\bc)\big)\Big \rangle_{U^d(G)}\Big|^{\frac{1}{2}},$$
For all $i=1,\cdots, d-1$, where $\pi_{i,j}(\bc) \in C_d$ is formed from $\bc$ by replacing  the $i^{\rm{th}}$ coordinates with $j$. 
Iterated this, we obtain the complete proof of Proposition \ref{Tao}.\\

Combining Cauchy-Bunyakovskii-Gowers-Schwarz inequality with the binomial formula and the multilinearity of the Gowers inner
product one can easily check that the triangle inequality for $\big|\big|.\big|\big|_{U^d(G)}$ holds.  We further have 
$$\big|\big|f\big|\big|_{U^d(G)} \leq \big|\big|f\big|\big|_{U^{d+1}(G)},$$
by applying Cauchy-Bunyakovskii-Gowers-Schwarz inequality with $f_{\bc}=1$ if $\bc_d=0$ and $f_{\bc}=1$ if $\bc_d=1$. 
The Gowers norms are also invariant under the shift and conjugacy.\\

If $G$ is a finite discrete Abelian group, we define the discrete derivative of function $f :G \rightarrow \C$ by putting
$$\partial_h(f)=f(x+h).\overline{f}(x),$$
for all $h,x \in G$. We can thus write the Gowers norm of $f$ as follows
$$\big|\big|f\big|\big|_{U^d(G)}^{2^d}=\int_{G^{d+1}}\partial_{h_1}\partial_{h_2}\cdots\partial_{h_d}(f)(x) d\bh dx.$$
If further $f$ take values on $\R/\Z$ and $\partial_{h_1}\partial_{h_2}\cdots\partial_{h_{d+1}}(f)(x)=0$, 
for all $h_1,\cdots,h_{d+1},x \in G$, then $f$ is said to be a polynomial function of degree at most $d$. The degree of $f$ is denoted by
$d^{\circ}(f)$.\\

According to this it is easy to see that for any function $f$ and any polynomial function $\phi$ of degree at most $d$, we have
$$\big|\big|e^{2 \pi i \phi(x)}f(x)\big|\big|_{U^d(G)}=\big|\big|f\big|\big|_{U^d(G)}.$$
Therefore
\begin{eqnarray}\label{vander}
 \sup_{\phi,~~d^{\circ}(\phi)\leq d}\Big|\int_G e^{2 \pi i \phi(x)}f(x)dx\Big| \leq \big|\big|f\big|\big|_{U^d(G)}.
\end{eqnarray}

In application and here we need to define the Gowers norms for a bounded function defined on $\big\{1,\cdots,N\big\}$. For that, 
if $f$ is a bounded function defined on $\big\{1,\cdots,N-1\big\}$, Following 
\cite{Green-Tao-Ziegler},  we put
$$ \big|\big|f\big|\big|_{U^d[N]}=
\frac{\big|\big|\widetilde{f}\big|\big|_{U^d(\Z/2^d.N\Z)}}{\big|\big|\1_{[N]}\big|\big|_{U^d(\Z/2^d.N\Z)}},$$
where $\widetilde{f}=f(x).\1_{[N]},$ $x \in \Z/2^d.N\Z$, $\1_{[N]}$ is the indicator function of $\big\{1,\cdots,N\big\}$.
For more details on Gowers norms, we refer to \cite{TaoHfourier},\cite{TaoAdditive}.\\

The sequence $f$ is said to have a small \linebreak Gowers norms if for any $d \geq 1$, 
$$\big|\big|f\big|\big|_{U^d[N]} \tend{N}{+\infty}0.$$
An example of sequences of small Gowers norms that we shall need here is given by Thue-Morse and Rudin-Shapiro sequences. 
This result is due J. Konieczny \cite{Konieczny}.\\

The Thue-Morse and Rudin-Shapiro sequences are a classical sequences both arise from a primitive substitution dynamical systems. 
The Thue-Morse and Rudin-Shapiro  sequences are defined respectively by  
$$ t(n)=e^{2 \pi i s_2(n)},$$
and 
$$r(n)=(-1)^{u_{11}(n)},$$
where $s_2(n)$ is the sum of digits of $n$ in base 2 and $u_{11}(n)$ is the number of ``$11$'' in the 2-adic representation of $n$.\\ 

Precisely, J. Konieczny proved the following
\begin{thm}[\cite{Konieczny}]The Gowers uniform norm of Thue-Morse sequence $t=(t_n)$ ( respectively the Rudin-Shapiro sequence $r=(r_n)$) satisfy: for any
$d \in \N^*$, there exists $c=c(d)>0$ such that $\big\|t\big\|_{U^d[N]}=O(N^{-c})$ ( respectively
$\big\| r\big\|_{U^d[N]}=O(N^{-c})$). 
\end{thm}

We thus have
\begin{Cor} The Thue-Morse and Rudin-Shapiro sequences are a higher order oscillating sequences.
\end{Cor}

We are now able to outline the proof of our main result \ref{Mainofmain}.\\

Consider the dynamical system generated by the Thue-Morse sequence or Rudin-Shapiro sequence and denote it respectively by $(X_t,S_t)$ and 
$(X_r,S_r)$.
Therefore, obviously, there is a continuous function $f \in C(X_i)$ and a point $y \in X_i$, $i=t,r$ such that 
$$\lim_{N \longrightarrow +\infty} \frac{1}{N}\sum_{n=1}^{N}f(S_i^ny) i(n)>0,~~~ for~~~i=t,r.$$

Notice that Sarnak's conjecture holds for the Thue-Morse sequence and Rudin-Shapiro sequence \cite{Tao-r},
\cite{Ab-Ka-Le}, \cite{Fe-Ku-Le-Ma},\cite{Ve} .\\

We end this paper by pointing out that, by applying inductively the standard \linebreak van der corput lemma \cite[p.25]{NK} or 
\eqref{vander}, 
on can easily obtain the following. 
\begin{thm} Assume that the Gowers norm of order $k$ of the bounded sequence $(c_n)$ is zero. 
Then, $(c_n)$ is an oscillating sequence of order $k$.
\end{thm}

This allows us to ask the following question.
\begin{Question}[\cite{elabalYe}] 
 Do we have that for any multiplicative function $(a(n))$ with small Gowers norm,  
 for any dynamical flow on a compact set $(X,T)$ with topological entropy zero, for any continuous function, for all $x \in X$,
\[
\frac1{N}\sum_{n=1}^{N}a_nf(T^nx) \tend{N}{+\infty}0?
\]
\end{Question}
Notice that by our assumption, it is obvious that for any nilsystem $(X,T),$  for any continuous function $f$, for any $x \in X$, we have
\[
\frac1{N}\sum_{n=1}^{N}a_nf(T^nx) \tend{N}{+\infty}0.
\]
Let us further point out that by the recent result of L. Matthiesen \cite{Lilian} there is a class of not necessarily bounded multiplicative 
functions with a small Gowers norms.
\begin{Question}[$\beta$-shift, $\times p$ maps and the M\"{o}bius-Liouville randomness law.]\label{Beta}
One may ask further on the validity of the  M\"{o}bius-Liouville randomness law for the map $T_px=px$ mod $1$, $p$ prime and for the beta-shift, that 
is, do we have 
$$\frac{1}{N}\sum_{n=1}^{N}\mob(n) f(T_p^nx) \tend{N}{+\infty}0, ~~~~~~~(\star)$$
$$\frac{1}{N}\sum_{n=1}^{N}\mob(n) f(T_\beta^nx) \tend{N}{+\infty}0,~~~\beta>1 ~~~~~~~(\star \star)$$
for $f$ a continuous function and $x \in [0,1)$?
\end{Question}
Let us notice that if we consider the doubling map $x \mapsto 2x$. Then, it is easy to see that $(\star)$ holds for a dense periodic points.
We furhter notice that the general case is related to the problem of mutiplicative order function $\ell(n)$. Whence, according 
to the main result in \cite{Banks}, for almost all integers $q$, for any $p$, we have 
$$\frac{1}{N}\sum_{n=1}^{N}\mob(n) \exp\Big({2 \pi i 2^n p \diagup q}\Big) \tend{N}{+\infty}0.$$
For a nice account on the  mutiplicative order problem we refer to \cite{Moree}.\\

Let us point out also that it is easy to see that for almost all $\beta>1$, $(\star \star)$ holds for a dense points.  
On the other hand it is well known that if $\beta$ is a simple Parry number then the map $T_\beta$ is measure theorically isomorphic 
to the subshift of finite type (SFT). But, by the recent result of D. Karagulyan \cite{Davit}, the SFT are not orthogonal to 
the M\"{o}bius function. We further have that if 
$\beta$ is a Parry number then the map $T_\beta$ is sofic, that is, a factor of SFT. But again by D. Karagulyan's result  \cite{Davit}
the M\"{o}bius function is not orthogonal to the sofic symbolic dynamical systems. However, this does not answer our question.\\

Finally, in a very recent work under progress, P. Kurlberg and the author proved that the answer to the question \ref{Beta} is negative. They further established that
the M\"{o}bius-Liouville randomness law does not hold for any expansive map \cite{AK}. Precisely, P. Kurlberg and the author proved that 
for given an integer $b \geq 2$,
  there exists $x \in [0,1)$ and $c > 0$ such that
$$
\sum_{n \le N} \lambda(n)  \sin( 2 \pi  b^{n } x)  \geq c \cdot N
$$
for {\em all} sufficiently large $N$.\\
This later result combined with Bourgain's result \cite{Sarnak} on the existence of dynamical system for which the M\"{o}bius-Liouville randomness law hold (see also \cite{DS}) allows us to ask the following
question.
\begin{Question} For any $\epsilon>0$, can one construct a dynamical system with topological entropy $\epsilon$ for which the 
 M\"{o}bius-Liouville randomness law hold.
\end{Question}

\begin{thank}
The author wishes to express his thanks to Fran\c cois Parreau 
for many stimulating conversations on the subject and his sustained interest and encouragement. He further wishes 
to express his thanks to Wen Huang for many interesting email exchanges on the subject related to this work. He would like also
to thanks Igor Shparlinski and Ping Xi for their interest and valuable comments. 
\end{thank}

\end{document}